\documentclass{article}

\pdfoutput=1

\usepackage{arxiv}
\usepackage[utf8]{inputenc} 
\usepackage[T1]{fontenc}    
\usepackage{hyperref}       
\usepackage{url}            
\usepackage{booktabs}       
\usepackage{amsfonts}       
\usepackage{nicefrac}       
\usepackage{microtype}      
\usepackage{lipsum}		
\usepackage{graphicx}
\usepackage{natbib}
\usepackage{multirow}
\usepackage{makecell}
\usepackage[normalem]{ulem}
\usepackage{doi}
\usepackage{amsmath,amssymb,amsfonts}
\usepackage{algorithm,algorithmic}
\usepackage{amsthm}
\newtheorem{theorem}{\indent Theorem}[section]
\newtheorem{lemma}[theorem]{\indent Lemma}

\newtheorem{remark}{\indent Remark}[section]

\title{On the robustness of double-word addition algorithms}

\author{ 
    Yuanyuan Yang\thanks{These authors contributed equally to this work.} \\
    Huawei Technologies Co., Ltd. \\
    yuanyuanyang@whu.edu.cn \\
    \And
    XinYu Lyu\footnotemark[1] \\
    School of Mathematics and Statistics, Wuhan University \\
    lvxinyu.math@whu.edu.cn \\
    \And
    Sida He \\
    Huawei Technologies Co., Ltd. \\
    hesida1@huawei.com \\
    \And
    Xiliang Lu \\
    School of Mathematics and Statistics and Hubei Key Laboratory of Computational Science, Wuhan University \\
    xllv.math@whu.edu.cn \\
    \And
    Ji Qi \\
    Huawei Technologies Co., Ltd. \\
    ryan.qiji@huawei.com \\
    \And
    Zhihao Li\thanks{Corresponding author} \\
    Huawei Technologies Co., Ltd. \\
    lizhihao.ict@huawei.com
}

\renewcommand{\shorttitle}{On the robustness of double-word addition algorithms}

\hypersetup{
pdftitle={A template for the arxiv style},
pdfsubject={q-bio.NC, q-bio.QM},
pdfauthor={David S.~Hippocampus, Elias D.~Striatum},
pdfkeywords={First keyword, Second keyword, More},
}

\begin{document}
\maketitle

\begin{abstract}
We demonstrate that, even when there are moderate overlaps in the inputs of sloppy or accurate double-word addition algorithms in the QD library, these algorithms still guarantee error bounds of $O(u^2(|a|+|b|))$ in faithful rounding. Furthermore, the accurate algorithm can achieve a relative error bound of $O(u^2)$ in the presence of moderate overlaps in the inputs when rounding function is round-to-nearest. The relative error bound also holds in directed rounding, but certain additional conditions are required. Consequently, in double-word multiplication and addition operations, we can safely omit the normalization step of double-word multiplication and replace the accurate addition algorithm with the sloppy one. Numerical experiments confirm that this approach nearly doubles the performance of double-word multiplication and addition operations, with negligible precision costs. Moreover, in directed rounding mode, the signs of the errors of the two algorithms are consistent with the rounding direction, even in the presence of input overlap. This allows us to avoid changing the rounding mode in interval arithmetic. We also prove that the relative error bound of the sloppy addition algorithm exceeds $3u^2$ if and only if the input meets the condition of Sterbenz's Lemma when rounding to nearest. These findings suggest that the two addition algorithms are more robust than previously believed.
\end{abstract}


\section{Introduction}

Floating-point numbers are fundamental in numerical computation, playing a crucial role in representing real numbers on computers \cite{higham2002accuracy,overton2001numerical}. They are extensively used in scientific computing, engineering, and finance to approximate real-world quantities and perform mathematical operations \cite{gustafson1997beating,monniaux2008certified}.

Managing the errors inherent in floating-point computations poses a significant challenge \cite{goldberg1991what}. Due to the finite precision of floating-point representations, operations involving these numbers can introduce rounding errors that accumulate and impact the accuracy of results. Understanding and mitigating these errors are essential for ensuring the reliability and correctness of numerical computations \cite{moler2004floating,hubbert2012understanding}.

To improve the accuracy of numerical computations, two fundamental algorithms have significantly influenced the management of rounding errors in floating-point addition: the 2Sum \cite{TAOCP,10.1007/BF01975722} and Fast2Sum \cite{Dekker1971Split} algorithms. These error-free transforms aim to compute the sum of two floating-point numbers with greater precision than the straightforward addition operation.

The IEEE 754 standard \cite{IEEE7542008} defines the format and behavior of floating-point numbers, ensuring a standardized and predictable computational environment across various platforms. Currently, IEEE 64-bit floating-point arithmetic is generally accurate enough for most scientific applications. However, an increasing number of critical scientific computing applications require higher numeric precision \cite{1425396}. To meet these demands, researchers have explored extended precision arithmetic, such as double-word arithmetic, often referred to as "double-double" in the literature. This approach involves representing a real number as the unevaluated sum of two floating-point numbers.

Since the floating-point format is typically implemented in hardware, double-word arithmetic offers performance advantages over software-emulated multiple precision methods.  Additionally, double-word arithmetic benefits from SIMD (Single Instruction, Multiple Data) instructions designed for the floating-point format \cite{9603410,10.1007/978-3-030-86976-2_14,DDAVX,9370307}, which can help minimize speed penalties.

Similar to 2Sum and Fast2Sum, two foundational addition algorithms exist in the realm of double-double data types: sloppy add (Algorithm \ref{algo:sloppyAdd}) and accurate add (Algorithm \ref{algo:accurateAdd}). These algorithms utilize the error-free transforms of 2Sum and Fast2Sum to achieve higher precision results. Notably, there is nearly a twofold difference in the number of floating-point operations used between sloppy add and accurate add.

In the analysis of errors in floating-point computations, considering error bounds \cite{wilkinson1963rounding} is crucial. These bounds establish an upper limit on the error that a computation may introduce, providing a quantitative assessment of result accuracy.   By estimating and bounding errors, researchers and practitioners can make informed decisions regarding the reliability of their computations and take appropriate measures to mitigate potential errors \cite{monniaux2008certified,olver2014numerical}.


For double-word addition,  \cite{JM2018DoubleDoubleError} advises against employing the sloppy add for adding two double-word numbers unless both operands have the same sign. This advice stems from the observation that current theoretical proofs do not always guarantee the relative error bound of sloppy add, unlike accurate add, which is known to provide correct results.

In practice, the sloppy add algorithm is widely adopted due to its higher efficiency compared to accurate add. Analyzing the error bound of the sloppy add algorithm is therefore crucial. In this study, we demonstrate that even in cases where there is moderate overlap among the inputs $(a_h, a_l)$ and $(b_h, b_l)$ in double-double addition algorithms, it is possible to establish a guaranteed error bound under the conditions of faithful rounding when using sloppy add. This finding provides a definitive conclusion to ensure the validity of sloppy add. 

Our analysis implies that replacing accurate add with sloppy add and omitting the normalization step of multiplication in the multiplication and addition operations is feasible. Additionally, our numerical experiments reveal that by making this substitution, the loss of precision is almost negligible, significantly enhancing the performance of double-double addition. Furthermore, we discovered that the two double-word addition algorithms are more robust than expected, with the rounding error being consistent with the rounding direction of floating-point operations, which is beneficial in interval arithmetic.

The structure of the article is outlined as follows. Section 2 introduces the symbols and background knowledge. Section 3 provide a detailed analysis and proof of the robustness of the double-double addition algorithm. Section 4 explores the application of the core ideas in double-double multiplication and interval arithmetic. Section 5 presents the numerical experiments.

\section{Notations and Preliminary}

In this paper, we operate under the framework of a radix-$2$, precision-$p$ floating-point (FP) arithmetic system, which we assume has an unlimited exponent range, thereby precluding any occurrences of underflow or overflow. Here, $p$ denotes the precision of the FP arithmetic system, and $\mathbb{F}_p$ signifies the set of all FP numbers of precision $p$.

Our analysis primarily concentrates on the round-to-nearest (RN) method and two directed rounding functions: RD (round toward $-\infty$) and RU (round toward $\infty$). A floating-point number $y$ is said to be a faithful rounding of a real number $x$ if $y \in \{\text{RD}(x), \text{RU}(x)\}$, which we denote as $y = \circ(x)$. It is evident that for any $x \in \mathbb{F}$, the round-to-nearest value $\text{RN}(x)$ is equal to $x$, as is the faithful rounding $\circ(x) = x$. When discussing directed rounding, we confine our consideration to RD and RU exclusively. The unit roundoff, represented by $u$, is defined as $u = 2^{-p}$ for RN rounding, and $u = 2^{1-p}$ when dealing with directed rounding or faithful rounding.

 
 
   For $x \in \mathbb{F}$, the notations $\text{ufp}(x)$, $\text{ulp}(x)$, and $\text{uls}(x)$ denote the unit in the first place, unit in the last place, and unit in the last significant place, respectively. The term $e_x$ represents the exponent bits of the floating-point (FP) number $x$. Consider, for instance, a double-precision FP number expressed as:
    \begin{equation}
    	x=(1.x_1x_2\cdots x_{49}100)_2\cdot 2^e = (1x_1x_2\cdots x_{49}100)_2 \cdot 2^{e-52},  
    \end{equation}
    we have, $\text{ufp}(x)=2^{e}$, $\text{ulp}(x)=2^{e-52}$, $\text{uls}(x)=2^{e-50}$, $e_x=e-52$.
    
    And one can easily check that  
\begin{equation*}
     2^{-p}|x|\leq\text{ulp}(x) \leq 2^{-p+1}|x|, |x-\text{RN}(x)|\leq u|x|.
\end{equation*}

We will frequently invoke the following Sterbenz Lemma to analyze the error bounds of certain algorithm

    \begin{lemma}[\cite{SterbenzLemma}]
	In a floating-point system with subnormal
	numbers available, if $x$ and $y$ are finite floating-point numbers such that
	\begin{equation}
		\frac y 2 \leq x \leq 2y,
	\end{equation}
	then $x-y$ is exactly representable as a floating-point number.
	\end{lemma}
Sterbenz Lemma shows that when $\frac y 2 \leq -x \leq 2y$, $\circ(x+y)=x+y$.

Another notation we will frequently use later is 
\begin{equation}
 r(x,y) = 
 \begin{cases}
 \frac {\min(|x|,|y|)} {\max(|x|,|y|)} & xy<0\\
 0 & \text{otherwise}
 \end{cases}.
\end{equation}
Obviously we have $0 \leq r(x,y) \leq 1$. $r(x,y)$ indicates the extent of cancellation when adding $x$ and $y$. When $\frac 1 2 \leq r \leq 1$, $x$ and $y$ meet the condition of Sterbenz Lemma.

Three fundmental algorithms: Fast2Sum, 2Sum, and 2Prod, are detailed below. The Fast2Sum algorithm was originally introduced by \cite{Dekker1971Split}.
    \begin{algorithm}[htb]
		\algsetup{linenosize=\small}
		\small
		\caption{Fast2Sum}\label{algo:Fast2Sum}
		\textbf{Input:} $a, b$. The floating-point exponents $e_a$ and $e_b$ satisfy $e_a\ge  e_b$\\
		\textbf{Output:} $ (s,t)$ s.t. $s+t=a+b$ 
		\begin{algorithmic}
			\STATE $ s \leftarrow \text{RN}(a+ b)$
			\STATE $ t \leftarrow \text{RN}( b-\text{RN}(s-a) )$ 
		\end{algorithmic}
	\end{algorithm}

The conventional 2Sum algorithm, due to \cite{TAOCP} and \cite{10.1007/BF01975722},	does not require $e_a\geq e_b$.	
	\begin{algorithm}[htb]
		\algsetup{linenosize=\small}
		\small
		\caption{2Sum}\label{algo:2Sum}
		\textbf{Input:} $ a,b $\\
		\textbf{Output:} $ (s,t)$ s.t. $s+t=a+b$
		\begin{algorithmic}
		\STATE $    s \leftarrow  \text{RN}(a+b)$
		\STATE $	a' \leftarrow \text{RN}(s-b) $
		\STATE $	b' \leftarrow \text{RN}(s-a') $
		\STATE $ 	\delta_a \leftarrow \text{RN}(a-a') $
		\STATE $ 	\delta_b \leftarrow \text{RN}(b-b')$
		\STATE $	t \leftarrow 	\text{RN}(\delta_a+\delta_b)$
		\end{algorithmic}
	\end{algorithm}

 In the widely used rounding-to-nearest, ties-to-even mode, the Fast2Sum and 2Sum algorithms produce results such that $s + t = a + b$. Conversely, in directed rounding mode, the following inequalities are observed:
\begin{equation*}
|s + t - (a + b)| \leq 2^{-2(p-1)}|a + b|,
\end{equation*}
and
\begin{equation*}
|s + t - (a + b)| \leq 2^{-2(p-1)}|s|.
\end{equation*}
Moreover, the rounding direction of $s + t$ in relation to $a + b$ is consistent with the prescribed floating-point rounding direction \cite{8742617}. Specifically, $s + t \geq a + b$ when rounding toward $+\infty$, and $s + t \leq a + b$ when rounding toward $-\infty$. Another significant feature is the Fast2Sum algorithm's ability to compute $s - a$ without any error\cite{6545904,10.1145/3054947}. Therefore, $t$ can be regarded as a faithful rounding of the expression $a + b - s$, satisfying the condition $|t| < \text{ulp}(s)$.

 
 
As most modern CPUs feature the fused multiply-add (FMA) instruction, it is advantageous to utilize FMA when computing the error in floating-point multiplication, as illustrated in Algorithm \ref{algo:2Prod} \cite{1570854174952231424,10.1145/641876.641878,HandbookFP2018}. In the absence of an FMA instruction, the most well-established method for calculating this error is the splitting algorithm \cite{Dekker1971Split}.
 
		\begin{algorithm}[htb]
		\algsetup{linenosize=\small}
		\small
		\caption{2Prod}\label{algo:2Prod}
		\textbf{Input:} $ a,b $\\
		\textbf{Output:} $ (s,t)$ s.t. $s+t=a\cdot b$  
		\begin{algorithmic}
			\STATE $ s \leftarrow \circ(a\cdot b)$
		    \STATE $ t \leftarrow \circ(a\cdot b-s)$ 
		\end{algorithmic}
	\end{algorithm}

Building upon the foundation established by the previously discussed algorithms, we are now in a position to present the two principal algorithms of this paper: sloppy add and accurate add.

Algorithm \ref{algo:sloppyAdd} was first described by \cite{Dekker1971Split}, albeit with a marginally different approach. It was subsequently implemented by Bailey within the QD library \cite{930115}, where it is referred to as sloppy addition. On the other hand, Algorithm \ref{algo:accurateAdd} has been incorporated into the QD library as IEEE addition. While Algorithm \ref{algo:sloppyAdd} can experience significant relative error due to the cancellation of $x_h$ and $y_h$—potentially as high as $1$—Algorithm \ref{algo:accurateAdd} maintains a relative error bound of $3u^2$ \cite{10.1145/2699469}.


		\begin{algorithm}[htb]
		\algsetup{linenosize=\small}
		\small
		\caption{sloppy addition algorithm}\label{algo:sloppyAdd}
		\textbf{Input:} $ (x_h,x_l), (y_h,y_l) $\\
		\textbf{Output:} $ (z_h,z_l) \approx (x_h,x_l)+ (y_h,y_l)$
		\begin{algorithmic}
			\STATE $ ( s_h, s_l)\leftarrow \text{2Sum}(x_h,y_h)$
			\STATE $ v \leftarrow \text{RN}(x_l+y_l)$
			\STATE $ w \leftarrow \text{RN}(s_l+v)$
			\STATE $ (z_h, z_l) \leftarrow \text{Fast2Sum}(s_h, w)$
		\end{algorithmic}
	\end{algorithm}

	\begin{algorithm}[htb]
		\algsetup{linenosize=\small}
		\small
		\caption{ accurate addition algorithm}\label{algo:accurateAdd}
		\textbf{Input:} $ (x_h,x_l), (y_h,y_l) $\\
		\textbf{Output:} $ (z_h,z_l) \approx (x_h,x_l)+ (y_h,y_l)$
		\begin{algorithmic}
			\STATE $ ( s_h, s_l)\leftarrow \text{2Sum}(x_h,y_h)$
			\STATE $ ( t_h, t_l)\leftarrow \text{2Sum}(x_l,y_l)$
			\STATE $ c \leftarrow \text{RN}(s_l+t_h)$
			\STATE $ (v_h,v_l)\leftarrow \text{Fast2Sum}(s_h,c) $
			\STATE $ w\leftarrow\text{RN}(t_l+v_l)  $
			\STATE $ (z_h, z_l) \leftarrow \text{Fast2Sum}(v_h, w)$
		\end{algorithmic}
	\end{algorithm}

\section{Robustness of Double-Word Addition Algorithms}

In this section, we will demonstrate that both Algorithm \ref{algo:sloppyAdd} and Algorithm \ref{algo:accurateAdd} remain robust even when the inputs do not strictly adhere to the non-overlap condition required for faithful rounding. To begin, we will establish one sufficient condition for the Fast2Sum algorithm.


	\begin{lemma}\label{lem:fast2sumextra}
Assume that the floating-point numbers $a, b \in \mathbb{F}_p$ fulfill the conditions $\text{ufp}(a) \leq \text{ufp}(b)$ and $\text{uls}(a) \geq \text{ulp}(b)$. When all three floating-point operations in the Fast2Sum algorithm are executed with faithful rounding, the equality $s + t = a + b$ is preserved. Furthermore, the rounding directions applied to each of the three floating-point operations may vary.

	\end{lemma}
	\begin{proof}
	Without loss of generality, we may assume $a =\pm M 2^e, b = N 2^e$, where $M,N \in \mathbb{N}, 1\leq M < 2^{p}, 2^{p-1}\leq N < 2^p$.
	
	When $a=-M2^e$,  we have
  \begin{equation*}
     a+b = (N-M) 2^e=\circ(a+b),
 \end{equation*}
 Hence  there is no rounding error for $s=\circ(a+b)$, $t=0$. 
	
When $a=M2^e$, one can find that
   \begin{equation*}
  a+b=(M+N)2^e, 2^{p-1}+1 \leq M+N < 2^{p+1}. 
  \end{equation*}
If $M+N<2^p$, then $a+b$ can be represent by floating number $s=(M+N)2^e$, and $a+b=s+t$.
 
If $ 2^p\leq M+N < 2^{p+1}$ and $M+N$ is even, then $s=(M+N)/2 \cdot 2^{e+1}$, $a+b=s+t$. 
	
Next we consider the case: 
\begin{equation*}
 2^p\leq M+N <2^{p+1}, M+N=2Q+1, Q\in\mathbb{N}. 
\end{equation*}

 If no carry happens when rounding $a+b$ to floating-point number $s$, then

 \begin{eqnarray*}
&& s=2Q2^e, Q<2^p,\\[1.5ex]
&& t=\circ(y-\circ(2Q2^e-M2^e)) = \circ(N2^e-(N-1)2^e )=2^e,\\[1.5ex]
&& a+b=(2Q+1)2^e=s+t. 
 \end{eqnarray*}
 
 If there is a carry when rouding $a+b$ to floating-point number $s$, then 

 \begin{eqnarray*}
 && s=(2Q+2)2^e,  Q+1 \leq 2^p,  \\[1.5ex]
 && t=\circ(y-\circ((2Q+2)2^e-M2^e))= \circ(N2^e-(N+1)2^e )=-2^e,\\[1.5ex]
 && a+b=(2Q+1)2^e=s+t.
 \end{eqnarray*}
 
	\end{proof}

 The condition specified in Lemma \ref{lem:fast2sumextra} is optimal in the following sense. Consider the case where $a = 2^{-1} + 2^{-p}$, $b = 1 + 2^{-(p-1)}$, and $\text{uls}(a) = \frac{1}{2} \text{ulp}(b)$. When the sum $s + t$ is computed using the Fast2Sum algorithm, the resulting relative error is of the order $O(u)$, as opposed to $O(u^2)$ or zero. This holds true across all standard rounding modes, including round-to-nearest (RN), round-down (RD), and round-up (RU).

Armed with Lemma \ref{lem:fast2sumextra}, we intend to validate the accuracy of the Fast2Sum implementations in both Algorithm \ref{algo:sloppyAdd} and Algorithm \ref{algo:accurateAdd}.
	
	
	
	\begin{theorem}\label{the:sloppyfast2sum}
Assume that $p \geq 6$ and let there be two numbers $o_x$ and $o_y$ residing within the interval $[1, \frac{1}{8u}-2]$. If the conditions $|x_l| \leq o_x u |x_h|$ and $|y_l| \leq o_y u |y_h|$ are satisfied, then the Fast2Sum operation in Algorithm \ref{algo:sloppyAdd} is guaranteed to be valid in faithful rounding, where the ``valid'' means either error-free or the relative error is $O(u^2)$. 
	\end{theorem} 

\begin{proof} Denote by $o=\max\{o_x,o_y\}$.  First we give a bound of $|w|$,
\begin{equation*}
	\begin{aligned}
	 |w| & \leq (|s_l|+|v|)(1+u) \\[1.5ex]
	      & \leq \left( u(|x_h|+|y_h|) + (o_xu|x_h|+o_yu|y_h|)(1+u)\right)(1+u) \\[1.5ex]
	      & \leq (|x_h|+|y_h|) (1+o(1+u))(1+u)u \\[1.5ex]
	      & <(o+2)u(|x_h|+|y_h|). 
\end{aligned}
\end{equation*}

If $|w|\leq |s_h|$ the Fast2Sum algorithm is error-free. Next we assume that $|w| > |s_h|$, then $x_h$ and $y_h$ must have different signs. Let $y_h = -(1-\alpha)x_h, 0 \leq \alpha \leq 1$, then
\begin{eqnarray*}
&&|s_h| \geq (1-u)|x_h+y_h| = \alpha(1-u) |x_h|,\\[1.5ex]
&&|x_h| + |y_h| = (2-\alpha)|x_h|.
\end{eqnarray*}

From $|w| > |s_h|$ we have 
\begin{equation*}
(o+2)u(2-\alpha)|x_h| > \alpha(1-u)|x_h| \Rightarrow 
\alpha < \frac{2u(o+2)}{1+u+ou} = \frac{1}{2} - \frac{1-3ou-7u}{2(1+u+ou)} < \frac{1}{2},
\end{equation*}
Which implies that $x_h$ and $y_h$ meet the condition in Sterbenz Lemma, thus $s_l = 0$. Moreover
\begin{equation*}
|s_h| = k \cdot \frac{1}{2} \text{ulp}(x_h), k \in \mathbb{N}, \Rightarrow \text{uls}(s_h) \geq \frac{1}{2} \text{ulp}(x_h) \text{ or } s_h = 0.
\end{equation*}

As
\begin{equation*}
|w| < u(o+2)(2-\alpha)|x_h|, 
\Rightarrow \text{ulp}(w) \leq 2u|w| < 2u^2(o+2)(2-\alpha)|x_h| \leq 4u(o+2) \text{ulp}(x_h). 
\end{equation*}
Then we obtain that
\begin{equation*}
\frac{1}{2} \text{ulp}(x_h) \geq 4u(o+2) \text{ulp}(x_h), o \leq \frac{1}{8u} - 2.
\end{equation*}
Hence the normalization step is error-free.
\end{proof}	


Similar as Theorem \ref{the:sloppyfast2sum}, we have
\begin{theorem}\label{the:accuratefast2sum}
Assume that $p \geq 6$ and let there be two numbers $o_x$ and $o_y$ residing within the interval $[1, \frac{1}{8u}-2]$. If the conditions $|x_l| \leq o_x u |x_h|$ and $|y_l| \leq o_y u |y_h|$ are satisfied, and $|t_l|\leq o_t u |t_h|, |o_t|\in [1,\frac 1 {8u} -2]$. then the Fast2Sum operation in Algorithm \ref{algo:accurateAdd} is guaranteed to be valid in faithful rounding.
\end{theorem}

\begin{proof}
One notices that the Fast2Sum step $(v_h,v_l)\leftarrow \text{Fast2Sum}(s_h,c)$ is same as $(z_h,z_l) \leftarrow \text{Fast2Sum}(s_h,w)$ in Algorithm \ref{algo:sloppyAdd}, which has been proved to be valid. So we only need to show the validness of the last normalization step. 

If $r(x_h,y_h)<\frac 1 2$. Assume $\alpha=\max(|x_h|,|y_y|)$, we have $2\alpha \geq |s_h|\geq \frac 1 2 \alpha$. $\max(|x_l|, |y_l|)\leq ou \alpha$. $|s_l|\leq 2u \alpha $. $|t_h|\leq 2ou\alpha$.
$|c|=\text{RN}(s_l+t_h)\leq \text{RN}(2 u  \alpha + 2ou \alpha)\leq 5ou\alpha$. $|v_h|\geq |\text{RN}(|s_h|-|c|)|\geq (\frac 1 2 - 5ou)\alpha$. $|c|<|s_h|$, $|v_l|\leq u|v_h|\leq u(2|s_h|) \leq 4u\alpha$. $|w|\leq \text{RN}(|t_l|+|v_l|)\leq \text{RN}(|t_h|+|v_l|) \leq \text{RN}((2ou+4u)\alpha) \leq |v_h|$.

If $r(x_h,y_h)\geq \frac 1 2$. Due to Sterbenz Lemma, we have $s_l=0$, $c=t_h$. The last three steps of Algorithm \ref{algo:accurateAdd} is in fact applying Algorithm \ref{algo:sloppyAdd} on $(s_h,0)$ and $(t_h,t_l)$. Again owing to Theorem \ref{the:sloppyfast2sum}, the last Fast2Sum in Algorithm \ref{algo:accurateAdd} is valid.

\end{proof}

\begin{remark}
Roughly speaking, when the input values to the Fast2Sum operations in Algorithm \ref{algo:sloppyAdd} or Algorithm \ref{algo:accurateAdd} do not meet the condition $|a| \geq |b|$, significant cancellation is likely to occur during the double-word addition. This substantial cancellation tends to produce a large number of trailing zeros in a floating-point number $a \in \mathbb{F}$, resulting in $\text{uls}(a)$ being considerably much larger than $\text{ulp}(a)$. However, even in cases where $|b| > |a|$, as long as the overlap of the inputs is within a moderate range, the condition $\text{ulp}(b) \leq \text{uls}(a)$ holds. By leveraging Lemma \ref{lem:fast2sumextra}, we can assert that the normalization steps within these algorithms will be valid.
\end{remark}

	Now we are going to discuss the absolute error and relative error of the two addition algorithms.
	
	Let's start with a brief review of error bounds of Algorithm \ref{algo:sloppyAdd} in round-to-nearest mode and there is no overlap in inputs. \cite{lauter2005tripledouble} gives relative error bounds of an algorithm similar with Algorithm \ref{algo:sloppyAdd}. The two algorithms are essentially the same except the order of  computation in summation $a_l+b_l+s_l$.
 
 Lauter proves when the $x_hy_h>0$, the relative error is no greater than $4u^2$, and when $x_hy_h <0$, the relative error is no greater than $16u^2$ as long as $r(x_h,y_h)\leq \frac 1 2$.
 
 \cite{10.1145/2699469} prove that when $x_hy_h>0$, Algorithm \ref{algo:sloppyAdd} has a relative error bound of $3u^2$, and \cite{Vicent2022DDNrm2} claim this bound is asymptotically optimal. 
 
 We will give a tight error bound of Algorithm \ref{algo:sloppyAdd} when the input satisfy $x_hy_h<0$ and $r(x_h,y_h)\leq \frac 1 2$.
 
	\begin{theorem}\label{thm:sloppyrel}
		In rounding to nearest mode, suppose $\text{RN}(x_h+x_l)=x_h$, $\text{RN}(y_h+y_l)=y_h$, $x_h$ and $y_h$ have opposite signs. If $r(x_h,y_h) \leq \frac 1 2$, Algorithm \ref{algo:sloppyAdd} 	has relative bound of $3u^2+O(u^3)$. Otherwise if the tie-breaking rule is ties-to-even or ties-to-away, except the case $|x_h|$ and $|y_h|$ are two consective FPNs, the relative error of Algorithm \ref{algo:sloppyAdd} is no greater than $u$. 
	\end{theorem}
	
	\begin{proof}
		
	Owing to the symmetry of $(x_h,x_l)$ and $(y_h,y_l)$ in Algorithm \ref{algo:sloppyAdd} and the symmetry of rounding function round-to-nearest and tie-breaking rule,  without loss of generality we may assume that
 \begin{equation*}
 x_h \geq - y_h > 0 .
 \end{equation*}
 Denote by 
 \begin{equation*}
 x_h=M_x\cdot 2^{e_x}, y_h=-M_y\cdot 2^{e_y},
  \end{equation*}
 \begin{equation*}
 2^{p-1}\leq M_x,M_y < 2^p, M_x,M_y\in\mathbb{N}, e_x,e_y\in\mathbb{Z},e_x\geq e_y.
  \end{equation*}
 
 Throughout the proof we will frequently use $\eta = \text{ulp}(y_h)$ as a unit. The sum of $x_h$ and $y_h$ is denoted by
 \begin{equation*}
 \Delta M \eta, \Delta M=M_x\cdot 2^{e_x-e_y} -M_y\in \mathbb{N}. 
 \end{equation*}
And we have
 \begin{equation*}
 r(x_h,y_h)=M_y/ (M_x 2^{e_x-e_y})=M_y/(M_y+\Delta M).
 \end{equation*}
 
 The relative error of addition of $x_l$ and $y_l$ is denoted by 
 \begin{equation*}
 \epsilon, |\epsilon|\leq u, i.e., \text{RN}(x_l+y_l)=(x_l+y_l)\cdot(1+\epsilon). 
 \end{equation*}
	
	Our goal is to proof that when $r(x_h,y_h)\leq \frac 1 2$, the relative error $\varepsilon$ is bounded by $3u^2+O(u^3)$. Now we discuss the relative error $\varepsilon$ accroding to the relation of $e_x$ and $e_y$. 
	
	\begin{itemize}
		\item $e_x=e_y$. 
		 It's easy to see that 
   \begin{equation*}
   r(x_h,y_h) =M_y /M_x > 2^{p-1}/2^p=1/2
   \end{equation*}
   
   when $e_x=e_y$.
		   
		 According to Sterbenz Lemma, we have $s_l=0$. 
   Since 
   \begin{equation*}
   \text{RN}(x_h+x_l)=x_h, \text{RN}(y_h+y_l)=y_h,  
   \end{equation*}
   
   it's clear that 
   \begin{equation*}
   |x_l|,|y_l|\leq \frac 1 2 \eta. 
   \end{equation*}
   
   And
   $|x_l| + |y_l| =\eta$ if and only if $|x_l| = |y_l| = \frac \eta 2$, in which case $ \epsilon  = 0$. 
   
   Thus the absolute error of algorithm is $|\epsilon(x_l + y_l)| < u\eta$, and the relative error is
		\begin{equation} \label{eq:errexeqey}
			\varepsilon =\left | \frac {\epsilon(x_l+y_l)} {s_h+x_l+y_l} \right | < \frac {u\eta} {(\Delta M \eta-(|x_l|+|y_l|))}.
		\end{equation}
		
		Now we are going to discuss $\varepsilon$ according the value of $\Delta M$. 
		\begin{itemize}
			\item 			
			$\Delta M=0$.  Clearly $s_h=0$, and $\varepsilon$ equals to the relative error of $\text{RN}(x_l+y_l)$, $|\varepsilon| \leq u$.
		
			\item 
			$\Delta M=1$. As $|x_l|+|y_l|\leq\eta$, it seems that the relative error can be infinity. However, $|x_l|+|y_l|$ cannot equal $\eta$. Noticing $\Delta M=1$, one of $M_x$ and $M_y$ must be odd and the other be even. According to the ties-to-even rule, only $|a_l|$ whose corresponding $a_h$ has an even last significant digit can be exactly $\frac 1 2 \text{ulp}(a_h)=\frac 1 2 \eta$, and the other can be as large as $(1-u)\frac 1 2 \eta$ at most, in which $a_h \in \{x_h,y_h\}, a_l\in\{x_l,y_l\}$. If the tie-breaking rule is ties-to-away and $|x_l|=|y_l|=\frac 1 2 \eta$, then $x_l=-\frac 1 2 \eta$ and $y_l=\frac 1 2 \eta$, $x_l+y_l=0$. So the relative error $\varepsilon$ remains bounded,
			\begin{equation}
				\varepsilon \leq \frac {u(\frac 1 2 \eta +(1-u)(\frac 1 2)\eta ) } {\eta-(\frac 1 2 \eta +(1-u)(\frac 1 2)\eta) } < 2.
			\end{equation}
			
			\cite{JM2018DoubleDoubleError} gives a counterexample of $\varepsilon=1$ where $e_x=e_y$ and $\Delta M=1$.
			
			\item $\Delta M \geq 2$. Recalling \eqref{eq:errexeqey}, 
			\begin{equation}
				\varepsilon <  \frac {u\eta} {\Delta M \eta - \eta}= \frac {u} {\Delta M -1} \leq  u.
			\end{equation}
		  
		  Since $e_x=e_y$, we have 
    \begin{equation*}
    2^{p-1}\leq M_y < M_x < 2^p. 
    \end{equation*}
    
    Assume that $M_x$ and $M_y$ follow discrete uniform distribution. Then the size of sample space is about$\frac 1 2 (2^{p-1})^2$, and points satisfying 
    \begin{equation*}
    \Delta M\geq 2^k(1\leq k \leq p-2) 
    \end{equation*}
    
    is about 
    \begin{equation*}
    \frac 1 2 (2^{p-1}-2^k)^2. 
    \end{equation*}
    
    So 
    \begin{equation*}
    P(\varepsilon \leq 2^{-p-k})\approx(1-2^{k-(p-1)})^2. 
    \end{equation*}
    
    In the case of double-precision, where $p = 53$. If $k = 40$, then $P(\varepsilon \leq 2^{-93})\approx 0.9995$; if $k = 48$, then $P(\varepsilon \leq 32u^2)\approx 0.8789$. In other words, even $e_x = e_y$ and cancellation happens in $x_h + y_h$, it is still possible to get a relative error of order $u^2$.
		\end{itemize}
		
		\item $e_x=e_y+1$.

		 First we give the bounds related to $x_l$ and $y_l$, 
   \begin{equation*}
   |x_l|\leq \frac 1 2 \text{ulp}(x_h)=\eta, |y_l|\leq \frac 1 2\eta, 
   \end{equation*}
   \begin{equation*}
   |x_l|+|y_l|\leq \frac 3 2 \eta, \text{RN}(|x_l|+|y_l)|\leq \frac 3 2 \eta.
   \end{equation*}
   
   Then the sum of $x_h+y_h$, 
   \begin{equation*}
   \Delta M \geq 2^{p-1}\cdot 2 - (2^p-1)=1,  s_h\geq \eta,
   \end{equation*}
   \begin{equation*}
   \Delta M\leq (2^p-1)\cdot 2-2^{p-1}=3\cdot2^{p-1}-2\leq 2 \cdot 2^p, 
   \end{equation*}
   
   as $\Delta M\eta =s_h+s_l$ and $s_h=\text{RN}(\Delta M \eta)$, we have $|s_l|\in\{0, \eta\}$.
		
		Again we are going to discuss $\varepsilon$ according the value of $\Delta M$. 
		\begin{itemize}
			\item 
			$\Delta M=1$. $\Delta M$ equals $1$ iff 
   \begin{equation*}
   M_x=2^{p-1},M_y=2^p-1, 
   \end{equation*}
   
   and obviously 
   \begin{equation*}
   s_l=0, w=v=\text{RN}(x_l+y_l). 
   \end{equation*}
   
   The relative error 
			\begin{equation}
				\varepsilon = \left | \frac {\epsilon(x_l+y_l)} {\eta + (x_l+y_l)} \right |,
			\end{equation}
			In this case $x_l+y_l$ can not be $-\eta$ either. First considering ties-to-even rule, Since 
   \begin{equation*}
   x_h=M_x2^{e_x}=2^{p-1}2^{e_x}, 
   \end{equation*}
   
   we have 
   \begin{equation*}
   - \frac 1 2 \eta\leq x_l \leq \eta .
   \end{equation*}
   
   If $x_l < - \frac 1 2 \eta$, then  $\text{RN}(x_h+x_l) \neq x_h$, which contradicts the nonoverlap condition $\text{RN}(x_h+x_l) = x_h$. As $M_y$ is odd, $|y_l|\leq (1-u)\frac 1 2 \eta$ because of the ties-to-even rule. If the tie-breaking rule is ties-to-away, we have 
   \begin{equation*}
   - \frac 1 2 \eta\leq x_l \leq \eta , 
   \end{equation*}
   
   and
   \begin{equation*}
   -(1-u)\frac 1 2 \eta \leq y_l \leq \frac 1 2 \eta. 
   \end{equation*}
   
   So in these two tie-breaking rules the relative error is still bounded,
			\begin{equation} \label{eq:errexe_y1M1}
				\varepsilon \leq \frac { u( \eta +(1-u)(\frac 1 2)\eta ) } {\eta-(\frac 1 2 \eta +(1-u)(\frac 1 2)\eta) }  <2,
			\end{equation}
			just like the case when $e_x=e_y, \Delta M=1$.
		    
			\item 
		    $\Delta M\geq 2$.
      To get the relative error bound, we only need to get the error bounds of floatint-point additions $x_l+y_l$ and $s_l+v$. It's trivial to see that 
\begin{equation*}
|v|=|\text{RN}(x_l+y_l)| \leq \text{RN}( 3\eta/ 2 )=3 \eta/2,
\end{equation*} 

and the roundoff error of $\text{RN}(x_l+y_l)$ is 
\begin{equation*}
|\epsilon(x_l+y_l)|\leq 1 /2\text{ulp}(3\eta /2 )= u\eta.
\end{equation*} 

Since $|w|\leq |\text{RN}(3\eta/ 2 +\eta)| = 5\eta /2 $, the roundoff error of $\text{RN}(s_l+v)$ is 
\begin{equation*}
|\text{RN}(s_l+v) -(s_l+v)|\leq 1/ 2 \text{ulp}(5\eta /2 )=2u\eta.
\end{equation*} 

Thus the relative error

		    \begin{equation}
			\begin{aligned}
				\varepsilon &\leq \left | \frac {|\epsilon(x_l+y_l)|+|\text{RN}(s_l+v)  -(s_l+v)|} {\Delta M \eta + (x_l+y_l)} \right | \\
				& =\frac {3u}  {\Delta M  - \frac 3 2 } 	
			\end{aligned}
			\end{equation}
			 
			 According to the inequality above, the error bound may be larger than $u$ when $2\leq \Delta M \leq 4$. However, in such case $s_l=0$, thus the roundoff error of $\text{RN}(s_l+v)$ is zero. It's easy to see that when $\Delta M\geq 3$, $\varepsilon \leq \frac 2 3 u$. Therefore, we only need to consider the case where $\Delta M=2$. Noticing that $e_x-e_y=1$, we can conlude that 
    \begin{equation*}
    M_x=2^{p-1}, M_y=2^p-2.
    \end{equation*}
    
    From $M_x=2^{p-1}$ we can deduce that 
    \begin{equation*}
     -\frac 1 2 \eta \leq x_l \leq \eta. 
    \end{equation*}
    
    In the case of  
    \begin{equation*}
    -\frac 1 2 \eta \leq x_l < 0, |x_l+y_l|\leq \text{RN}(\eta/2+\eta /2)=\eta. 
    \end{equation*}
    
    It seems that roundoff error $\epsilon$ can be $u\eta$. But when $|x_l+y_l|=\eta$, the addition is error-free, even if $\text{RN}(x_l+y_l)=\eta$ because of roundoff, the roundoff error is at most $\frac 1 2 u\eta$ rather than $u\eta$, otherwise if $\text{RN}(x_l+y_l)<\eta$, the roundoff error is clearly bounded by $\frac 1 2 u \eta$. Such that $\varepsilon\leq \frac {u\eta/2} {2\eta -\eta}=\frac u 2$. In the case of $0\leq x_l \leq \eta$, the roundoff error of $\text{RN}(x_l+y_l)$ is bounded by $u\eta$, and $\varepsilon \leq \frac {u\eta} {2\eta-\frac 1 2 \eta}=\frac 2 3 u$. In summary, it still holds that $\varepsilon \leq u$.
		\end{itemize}
		 
		 Now we focus on the sence that the cancellation is not severe, more specifically, 	
   \begin{equation*}
r(x_h,y_h) = \frac {M_y} {M_y + \Delta M}\leq \frac 1 2.
\end{equation*} 
   Obviously $\Delta M\geq M_y \geq 2^{p-1}$. 
   
   If $2^{p-1}\leq \Delta M < 2^p$, then $x_h+y_h=\Delta M\eta$ is a floaint-point number, consequently $s_l=0$. We only need to take in account the roundoff error of $\text{RN}(x_l+y_l)$, and the relative error 
   \begin{equation*}
   \varepsilon \leq \frac u {2^{p-1}-3/2}=2u^2+6u^3+\cdots. 
   \end{equation*}
   
   If $\Delta M\geq 2^p$,
   \begin{equation*}
   \varepsilon \leq \frac {3u} {2^p-3/2}=3u^2+\frac 9 2 u^3+\cdots. 
   \end{equation*}
   
   In summary, 
   \begin{equation*}
   \varepsilon \leq 3u^2+\frac 9 2 u^3+\cdots.
   \end{equation*}
   
   as long as $r(x_h,y_h)\leq \frac 1 2$. Moreover, this bound is a tight. Let $p=53$, consider 
		 \begin{align*}
		 	 x_h &= 844424930131969/2^{49},\\
		 	 x_l &= 2^{-53}, \\
		 	 y_h &= -4503599627370499/2^{53},\\ 
		 	 y_l &= 4714705859903487/2^{152},
		 \end{align*}
		 the relative error of Algorithm \ref{algo:sloppyAdd} is $2.99999999999998\cdots u^2$.
		 
		Similar to the analysis with the case of $e_x = e_y$, we assume that $M_x$ and $M_y$ follow discrete uniform distribution. The size of sample space is about ${2^{p-1}}^2$, number of points satisfying
  \begin{equation*}
  \Delta M \geq 2^k(0\leq k \leq p-1)
  \end{equation*}
  
  is about 
  \begin{equation*}
  {2^{p-1}}^2-\frac 1 2 2^k\cdot 2^{k-1}, 
  \end{equation*}
  
  and the probability of $\varepsilon \leq 3\cdot 2^{-p-k}$ is approximately $1-2^{2(k-p)}$. Let $p=53$, we obtain \begin{equation*}
P(\varepsilon\leq 3\cdot 2^{-94})\approx 1-2^{-24}= 1-5.96\cdots e-8(k=41),
\end{equation*}
\begin{equation*}
P(\varepsilon \leq 6u^2 ) \approx 1-2^{-2}= 0.75(k=p-1).
\end{equation*}
		
		\item $e_x > e_y+1$.  
In this case we always have $r(x_h,y_h)=M_y/M_x < \frac {2^p} { 4\cdot 2^{p-1}}=\frac 1 2$.

		Let $\xi = \text{ulp}(x_h)$, and  $x_h+y_h =c \xi$, then 
  \begin{equation*}
  x_h+y_h > (2^{p-1} - 2^{p-(e_x-e_y)} )\xi,
  \end{equation*}
  
  it follows that 
  \begin{equation*}
   2^{p-2} <  c < 2^p-1.
  \end{equation*}
  
  And  it's straightfoward to show that 
  \begin{equation*}
  |s_h|=\text{RN}(c)\xi. |s_l|\leq \frac 1 2 \text{ulp}(c) \xi.
  \end{equation*}
  
  Ovbiously 
  \begin{equation*}
		|s_l|\leq \frac 1 2\text{ulp}(c)\xi, |x_l|\leq \frac 1 2 \xi. |y_l|\leq \frac 1 2 \text{ulp}(y_h) \leq \frac 1 8 \xi. 
  \end{equation*}
  
  The relative error
		\begin{equation}
		\begin{aligned}
			\varepsilon &\leq  \frac {\frac  1 2 \text{ulp}(|x_l|+|y_l|)+ \frac 1 2 \text{ulp}(|s_l|+(1+\epsilon)(|x_l|+|y_l|))}  {c\xi-|x_l|-|y_l|}  \\
			&\leq \frac {\frac  1 2 \text{ulp}(\frac 5 8 \xi)+ \frac 1 2 \text{ulp}(\frac 1 2 \text{ulp}(c) \xi+(1+u)(\frac 5 8 \xi))}  {c\xi-\frac 5 8 \xi}   \\
			& = \frac  {\frac  1 2 u + \frac 1 2 \text{ulp}(\frac 1 2 \text{ulp}(c)+ \frac 5 8 (1+u))}  {c-\frac 5 8} 	
		\end{aligned}
		\end{equation}
		When $c\in(2^{p-2},2^{p-1})$, $\frac 1 2 \text{ulp}(c)=\frac 1 4$,
		\begin{equation} 
		\begin{aligned}
			\varepsilon & \leq \frac { \frac 1 2 u + \frac 1 2 \text{ulp} (\frac 1 4 + (1+u)(\frac 1 2 + \frac 1 8 ))} {c - \frac 5 8} \\
			& = \frac u {c-\frac 5 8} \\
			& \leq 4u^2+10u^3+\cdots.
		\end{aligned}
		\end{equation}
		The $4u^2$ in above inequality consists of $2u^2$ from $\text{RN}(x_l + y_l)$ and $2u^2$ from $\text{RN}(s_l + v)$. However, the contribution of the two additions can't be $2u^2$ at the same time.
  
  If $e_x-e_y > 2$, $c > 2^{p-1}-2^{p-3}=\frac 3 2 2^{p-2}$, thus 
  \begin{equation*}
  \varepsilon\leq \frac 8 3 u^2+\frac {40} 9 u^3+\cdots. 
  \end{equation*}
  
  Otherwise we have 
  \begin{equation*}
  e_x-e_y = 2, |x_h+y_h|= 4c\eta, 2^p<4c<2^{p+1}. 
   \end{equation*}
   
   It follows that 
   \begin{equation*}
   |x_h+y_h|=\lfloor 2c \rfloor \cdot2\eta + k\eta, k\in\{0,1\}, |s_l|\in\{0,\eta\}.
   \end{equation*}
   
   Next we will discuss the bound of $v$. Since $|x_l|\leq \frac 1 2\xi=2\eta$, $|y_l|\leq \frac 1 2 \eta$, we have 
   \begin{equation*}
   |v|=|\text{RN}(x_l+y_l) |\leq \text{RN} (\frac 5 2 \eta) = \frac 5 2 \eta. 
   \end{equation*}
   
   If $|v| < 2\eta$, the roundoff error of $\text{RN}(x_l+y_l)$ is
   \begin{equation*}
   \epsilon|x_l+y_l|\leq \frac  1 2 \text{ulp}(v)=\frac 1 4 u\xi,
   \end{equation*}
   
   and its contribution to $\varepsilon$ is bounded by $u^2$ instead of $2u^2$. So 
   \begin{equation*}
   \varepsilon \leq 3u^2+\frac {15} 2 u^3+\cdots.
   \end{equation*}
   
   Otherwise if $2\eta\leq |v| \leq \frac 5 2 \eta$, noticing that the significant bit of weight $\eta=\frac 1 2\text{ufp}(v)$ is always $0$ and $|s_l|\in \{0,\eta\}$, thus the addition of $s_l+v$ is error-free, and $\varepsilon\leq 2u^2+5u^3+\cdots$.
		
		To sum up, when $c\in(2^{p-2},2^{p-1})$, $\varepsilon \leq 3u^2+\frac {15} 2 u^3+\cdots$. 
		
		When $c \in [2^{p-1}, 2^p-1)$, $\frac 1 2 \text{ulp}(c)=\frac 1 2$.
		\begin{equation}
		\begin{aligned}
			\varepsilon & \leq \frac { \frac 1 2 u + \frac 1 2 \text{ulp} (\frac 1 2+ (1+u)(\frac 1 2 + \frac 1 8 ))} {c-\frac 5 8} \\
			& \leq \frac {\frac 3 2 u} {c-\frac 5 8} \\
			& \leq 3u^2+\frac {15} 4u^3+\cdots.
		\end{aligned}
		\end{equation}

       The error bound $3u^2$ is tight. Let $p=53$, consider
       \begin{align*}
       x_h &= 6755399441055745/2^{52},\\
       x_l &= 140737488355327/2^{100},\\
       y_h &=-4503599627370489/2^{53},\\
       y_l &= 4714705859903487/2^{152},
       \end{align*}
       the relative error of Algorithm \ref{algo:sloppyAdd} is $2.999999999999982\cdots u^2$.
	\end{itemize}	
\end{proof}

\begin{remark}
 $r(x_h,y_h)$ is an indicator of cancellation extent, from the proof of above theorem, we may find When $r(x_h,y_h) \leq \frac 1 2$, a tight relative error bound of $3u^2$ can be ensured. Otherwise if the tie-breaking rule is ties-to-even or ties-to-away, the relative error can not be large than $u$ if $\Delta M \neq 1$. Only when $\Delta M = 1$, i.e., $|x_h|,|y_h|$ are two consecutive floating-point numbers. the relative error may be in order $O(1)$.
\end{remark}

	Next we are to estimate the error bounds of the accurate addition algorithm \ref{algo:accurateAdd} in the presence of either  round-to-nearest or in faithful rounding.

	\begin{theorem}\label{thm:relerrAccurateFaith}
		Assume $p\geq 6$, if $|x_l|\leq o_x u|x_h|$, $|y_l|\leq o_y u|y_h|$, $1\leq o_x, o_y\leq \frac 1 {8u}-2$. Denote by 
  $o = \max\{o_x,o_y\}$, then the relative error of Algorithm \ref{algo:accurateAdd} is bounded by $(3o+15)u^2+O(u^3)$ if the rounding function is RN. If the rounding function is one of $\{\text{RD},\text{RU}\}$, the above bound is still true if the step  $(t_h,t_l)\leftarrow (x_l,y_l)$ in Algorithm \ref{algo:accurateAdd} is replaced with comparsion and Fast2Sum, where the signs of $x_l, y_l$ are consitent with the rounding direction, i.e., $x_l\geq 0,y_l\geq 0$ when RD and $x_l\leq 0 ,y_l\leq 0$ when RU.
	\end{theorem}

In the proof we need the following Lemma.
	\begin{lemma}\label{lem:boldo}[\cite{10.1145/3054947}]
	Let $a$ and $b$ be two floating-point numbers, and $s \in \{\text{RD}(a +
	b),\text{RU}(a + b)\}$. If the difference $|e_a-e_b|$ of the exponents of $a$ and $b$ does not exceed $p-1$, then	$s-(a + b)$ is a floating-point number.
	\end{lemma}
 
	Different from the round-to-nearest, 2Sum and Fast2sum may introduce a relative error of $u^2$ in directed rounding \cite{8742617}, and they are no longer error-free transformations literally. Immediately we know the absolute errors introduced by 2Sum and Fast2Sum are bounded by $2u^2(|x_h|+|y_h|)$ and $3u^2(|x_h|+|y_h|)$ for Algorithm \ref{algo:sloppyAdd} and Algorithm \ref{algo:accurateAdd}, respectively. Apart from Fast2Sum and 2Sum, the other floating-point operations have an error bound of $(2o+1)u^2(|x_h|+|y_h|)$ and $(o+2)u^2(|x_h|+|y_h|)$. 
	
	If $r(x_h,y_h) \leq \frac 1 2$, in other words, there is no severe cancellation, an absolute error bounded $\varepsilon$ implies a relative error bound of $3\varepsilon$. Obviously the relative error bounds of the two algorithms are $(6o+9)u^2$ and $(3o+15)u^2$.

It remains to estimate the relative errors when $r(x_h,y_h)> \frac 1 2$. 
 From Sterbenz Lemma, $x_h+y_h$ is a floaing-point number, and $\text{2Sum}(x_h,y_h)$ is error-free. Anyhow, the relative error of the sloppy addition Algorithm \ref{algo:sloppyAdd} suffers from cancellation because this algorithm drops the roundoff error of $\circ(x_l+y_l)$, we can't expected a satisfactory relative error bound.
 
 
 In accurate addition Algoritm \ref{algo:accurateAdd}, $\text{2Sum}(x_h,y_h)$ is error-free, but $\text{2Sum}(x_l,y_l)$ may not. The error introduced by  $\text{2Sum}(x_l,y_l)$ is of $u^2(|x_l|+|y_l|)\leq u^3(|x_h|+|y_h|)$, and $|x_h+x_l+y_h+y_l|$ can be of order $O(u^2)(|x_h|+|y_h|)$, thus we can't provied a relative error bound of $O(u^2)$ any more. 
 
 For example, let $p=53$, $x_h=2^{52}$, $x_l=1-2^{-53}$, $y_h=-(2^{52}+1)$, $y_l=2^{-107}$, the relative error of Algorithm \ref{algo:accurateAdd} is about $2^{-54}$ in RD mode. 
	
	However, we can try to replace $\text{2Sum}(x_l,y_l)$ with a comparsion and Fast2Sum, to get the desired error bound.
 
\begin{proof}
We will mainly prove the error bound for directed rounding.

The case of $r(x_h,y_h) \leq \frac 1 2$ is concluded in the above discussion. Next we will consider the directed rounding, by replacing $\text{2Sum}(x_l,y_l)$ with a comparsion and Fast2Sum. Without loss of generality we assume $|x_l|\geq |y_l|$. According the Lemma \ref{lem:boldo}, when the rounding function is RD or RU, Fast2Sum is error-free when $e_a-e_b\leq p-1$.

	When $e_{x_l}-e_{y_l} \geq p$. If the rounding function is RD, since $x_l\geq 0$, $y_l\geq 0$. If follows from $e_{x_l}-e_{y_l}\geq p$ that $\text{RD}(x_l+y_l)=x_l$, then we obtain $s=0$ and $t=y_l$, the Fast2Sum is error-free. It also holds true for RU. 
	
	
	The last Fast2Sum contributes a relative error bounded by $u^2$ to the final result, so  we only need to consider the error from $\text{Fast2Sum}(s_h,c)$. 
 
 From Lemma \ref{lem:boldo}, $\text{Fast2Sum}(s_h,c)$ is not error-free only when 
 \begin{equation*}
 e_{s_h}-e_c \geq p.
 \end{equation*}
 
 Which implies
 \begin{equation*}
 |t_h|=|c|\leq 2^{-(p-1)}|s_h|.
 \end{equation*}
 
 Notcing that $u=2^{-(p-1)}$ in directed rounding, $|t_h|\leq u|s_h|$. It follows that 
 \begin{equation*}
 (1-u) |s_h|\leq |v_h| \leq(1+u) |s_h|, 
 \end{equation*}
 \begin{equation*}
 |t_l|\leq u|t_h|\leq u^2|s_h|, 
 \end{equation*}
 \begin{equation*}
 |v_l|\leq u|v_h|\leq (u+u^2)|s_h|. 
 \end{equation*}
 
 Therefore the roundoff error of $\circ(t_l+v_l)$ is bounded by $(u^2+2u^3)|s_h|$, the error of  $\text{Fast2Sum}(s_h,c)$ is also bounded by $u^2|v_h|\leq u^2(1+u)|s_h|$. 
 
 The absolute error is bounded by  $(2u^2+O(u^3))|s_h| $ in total, and 
 \begin{equation*}
 |v_h+w|\geq |v_h|-(1+u)(|v_l|+|t_l|)\geq (1-(1+u)(u+2u^2))|s_h|. 
  \end{equation*}
 
 Adding the relative error $u^2$ from the normalization step, Algorithm \ref{algo:accurateAdd} still have a relative error bound $3u^2$ in this case .
 
In summary, when $(v_h,v_l)\leftarrow \text{Fast2Sum}(s_h,c)$ is not exact, the error of $\text{Fast2Sum}(s_h,c)$ and roundoff error of $\circ(t_l+v_l)$ is bounded by $u^2|s_h|$, and the sum is of order $|s_h|$, thus the relative error of Algorithm \ref{algo:accurateAdd} is of order $u^2$; and the last normalization is safe even not exact.

  If  $(v_h,v_l)\leftarrow \text{Fast2Sum}(s_h,c)$ is exact, we only need to check the influence of roundoff error of $\circ(t_l+v_l)$ in the presence of overlap in inputs. The contribution to the relative error of final result of $\circ(t_l+v_l)$ is
	\begin{equation} \label{eq:errAccurateFaith}
	\begin{aligned}
		\epsilon &\leq \frac {u(|t_l|+|v_l|)} {|x_h+y_h+x_l+y_l|} \\
		         &\leq u \frac {|t_l|+u|s_h+t_h|} {|s_h+t_h+t_l|} \\
		         & \leq u \frac {u|s_h+t_h|-u|t_l|+(1+u)|t_l|} { \left| |s_h+t_h| - |t_l|\right| }\\
		         &\leq u^2+u(1+u)\frac {|t_l|} { \left| |s_h+t_h| - |t_l|\right| }.
	\end{aligned}
	\end{equation}
	If $r(s_h,t_h)> \frac 1 2$, $|s_h+t_h|\geq \frac 1 2 |t_h| \geq \frac 1 {2u}|t_l|$. Plug back into \eqref{eq:errAccurateFaith} we obtain $\epsilon\leq 3u^2+6u^3+\cdots$. Otherwise due to Sterbenz Lemma, $v_l=0$, there is no roundoff error in $\circ(t_l+v_l)$. Taking the error of the normalization into account, the relative error of Algorithm \ref{algo:accurateAdd} is no greater than $4u^2$. 

 If the rounding function is RN, all Fast2Sums and 2Sums are error-free, hence the error only comes from $\text{RN}(t_l+v_l)$, and it's bounded by $3u^2+O(u^3)$ just as in the case of RD or RU.

In conclusion, when $r(x_h,y_h)>\frac 1 2$, Algorithm \ref{algo:accurateAdd} has a relative error bound of $4u^2+O(u^3)$. Combine with the relative error bound $(3o+15)u^2+O(u^3)$ when $r(x_h,y_h)\leq \frac 1 2$, the relative error bound of Algorithm \ref{algo:accurateAdd} is $(3o+15)u^2+O(u^3)$.
	\end{proof}

	The advantage of the accurate addition algorithm, i.e., Algorithm \ref{algo:accurateAdd}, is it's the relative error bound is insensitive to the cancellation of input even when there is overlap in inputs and the rounding function is not RN, which ensures that reasonable result can be obtained when there is no perturbation in the input.  However, such case is rare as input of almost all functions are the output of other functions. An exception is the functions in mathematical library, where the input is expected to be precise. For example, in the implementation of elementary functions in the QD library, the evaluation of Taylor expansions succeeding additive range reduction can safely use sloppy addition algorithm, while the additive range reduction itself cannot. 
 
 Take for example summation, dot product ans polynomial evaluation. Suppose FMA is unavailable, regardless of the error bound of addition algorithm is $\varepsilon(|x|+|y|)$ or $\varepsilon|x+y|$, the error bounds of ordinary recursive summation, recursive dot product, or polynomial evaluation(Horner rule) are 
	\begin{equation*}
	\left|\text{OrdinaryRecursiveSum}(a)-\sum_{i=1}^n a_i\right| \leq	\gamma_{n-1} \sum_{i=1}^n|a_i|,
	\end{equation*}
	\begin{equation*}
    \left|\text{RecursiveDotProduct}(a,b)-\sum_{i=1}^n(a_i\cdot b_i)\right| \leq \gamma_n \sum_{i=1}^n |a_i\cdot b_i|,
	\end{equation*}
	\begin{equation*}
    \left| \text{Horner}(p,x)-p(x) \right| \leq  \gamma_{2n}\sum_{i=0}^n |a_i|\cdot |x|^i.
	\end{equation*}
  In which $\gamma_n$ is defined as $\gamma_n \triangleq \frac {n \varepsilon} {1-n\varepsilon}$\cite{HandbookFP2018}. This implies at least using Algorithm \ref{algo:sloppyAdd} in place of Algorithm \ref{algo:accurateAdd} in BLAS functions based on summation and dot product won't result in much worse error. The only exception is summation of two numbers.

  \section{Applications in Double-Word Algorithms and Interval Arithmetic} 	

 \subsection{Double-word Multiplication}
We will discuss two applications in this section, i.e., double-double multiplication and interval arithmetic. There are many varies of multiplication algorithms for double-double types \cite{JM2018DoubleDoubleError}, and they can be summarized in a unified form (Algorithm \ref{algo:unifymult}).
	
	\begin{algorithm}[htb]  
		\algsetup{linenosize=\small}
		\small
		\caption{unified double-word multiplication}\label{algo:unifymult}
		\textbf{Input:} $ (x_h,x_l), (y_h,y_l) $\\
		\textbf{Output:} $ (z_h,z_l) \approx (x_h,x_l)\times (y_h,y_l)$
		\begin{algorithmic}
			\STATE $ ( c_h, c_l)\leftarrow \text{2Prod}(x_h,y_h)$
			\STATE $ t \leftarrow \text{sum}(c_l,x_hy_l, x_ly_h, x_l y_l)$, and $x_ly_l$ is optional. 
			\STATE $ (z_h, z_l) \leftarrow \text{Fast2Sum}(c_h, t)$
		\end{algorithmic}
	\end{algorithm}
	
	The difference of the multiplication algorithms lies in the strategies of computing summation $c_l+ x_hy_l+ x_ly_h+ x_l y_l$ and whether to ignore the $x_ly_l$ term. 
	
Since 
 \begin{align*}
     |x_l| \leq u |x_h|, \quad      |y_l| \leq u|y_h|,
 \end{align*}
  we have 
  \begin{align*}
|x_hy_l| \leq u|x_hy_h|,\quad |x_ly_h| \leq u|x_hy_h|,\quad 
|x_ly_l|\leq u^2 |x_hy_h|,
  \end{align*}
 provided  with $|c_l|\leq u |c_h|$. Then we can find that $|t|\leq (3u+O(u^2))|c_h|$. 
	
When the normalization step  $ (z_h, z_l) \leftarrow \text{Fast2Sum}(c_h, t)$ is removed and the multiplication is followed by double-word addition, i.e., $(x_h,x_l)$ is the unnormalized product, then we can obtain $|x_l|\leq (3u+O(u^2))|x_h|$ instead of $|x_l|\leq u|x_h|$.

 Given the substantial performance overhead, a combination of multiplication and addition operations, hereafter referred to as MAA (Multiplication and Addition), is often preferred over a double-word Fused Multiply-Add (FMA). The FMA operation is distinct in that it involves only one rounding step, whereas MAA involves separate rounding steps for each operation. It is evident that within the MAA, the error produced by the multiplication operation propagates into the subsequent addition. Consequently, the advantage of using an accurate addition algorithm for reducing error is negated.

As we have previously demonstrated, both sloppy and accurate addition algorithms are capable of correctly handling cases of moderate input overlap. Therefore, it is possible to achieve enhanced performance with MAA by omitting the normalization step in multiplication and utilizing sloppy addition instead.

This approach has been implemented in the mathematical library SLEEF \cite{SLEEF2020}, particularly when evaluating approximation polynomials. SLEEF adopts an even more aggressive strategy by eliminating the normalization step in the addition too, effectively employing the pair arithmetic method as proposed by \cite{Rump2020PairNIC}. Here, we provide a succinct rationale for its efficacy. Prior to the evaluation of approximation polynomials for elementary functions, inputs are typically reduced to a relatively narrow range, where the absolute values are considerably smaller than $1$ \cite{HandbookFP2018}. Additionally, the absolute values of coefficients in Taylor expansions—and by extension, approximation polynomials—tend to decrease monotonically with increasing order, provided the polynomial degrees are not excessively high. Consequently, significant cancellation is unlikely to occur during MAA in the evaluation of these polynomials, ensuring that $|s_h| \gg |w|$ in Algorithm \ref{algo:sloppyAdd}. In essence, the process of evaluating approximation polynomials is well-conditioned.

	In general, the normalization step of Algorithm \ref{algo:sloppyAdd} cannot be omitted, as the extent of cancellation in addition or the satisfaction of the NIC condition proposed by \cite{Demmel2008NIC} is unknown. It's clear that pair arithmetic has better performance at the expense of error bound\cite{Rump2020PairNIC}. Though MAA as we proposed in this paper requires more floating-point operations, the error bound is closer to the double-word algorithm with full normalizations. Numerical experiments in following section will show that the precision loss is negligible.

We now turn our attention to estimating the error of the proposed MAA (Multiplication-Addition Approach) algorithm. Let the exact mathematical result we aim to compute be $a \cdot b + c$, and let $d$ represent the result obtained using MAA. For the purposes of our analysis, we redefine the relative error as
 $$\frac {|d-(a\cdot b+c)|} {|a\cdot b|+|c|}.$$
The justification for this modification is as follows: The precise product of $a \cdot b$ must be truncated to fit within a double-word representation. Consequently, the conventional relative error measure, $\frac{|d - (a \cdot b + c)|}{|a \cdot b + c|}$, can become unbounded, irrespective of the accuracy of the multiplication and addition algorithms employed.

We define a similar modified relative error for the addition of $x$ and $y$ as $\frac{|d - (x + y)|}{|x| + |y|}$, where $d$ is the computed result of the addition algorithm. It is important to note that this modified relative error for addition does not apply to multiplication, where the standard relative error remains applicable.
	
	Assume the modified relative error bounds of multiplication and addition on are $\epsilon_{mul}$ and $\epsilon_{add}$ respectively, then the modifed relative error of MAA is
	\begin{equation}\label {eq:MAArelerr}
	\begin{aligned}
			\epsilon &= \frac  { |d - (a\cdot b+c) |} { |a\cdot b|+|c| } \\
			&\leq \frac { |a\cdot b|\epsilon_{mul} + (|a\cdot b|(1+\epsilon_{mul}) + |c|)\epsilon_{add}  }   { |a\cdot b|+|c|} \\
			& \leq \epsilon_{add}+\epsilon_{mul}(1+\epsilon_{add}) \\
			&\approx \epsilon_{add}+\epsilon_{mul}.
	\end{aligned}
	\end{equation}
 	
	The modified relative error bound of double-word multiplication $\epsilon_{mul}$ given in \cite{JM2018DoubleDoubleError} is $4u^2\sim 5u^2$, and the modified relative error bounds of two double-word addition algorithms Algorithm \ref{algo:sloppyAdd} and Algorithm \ref{algo:accurateAdd} are both no greater than $3u^2$\cite{JM2022DoubleDoubleErrorFormal}. If follows from \eqref{eq:MAArelerr} that when the normalization in multiplication is not skipped, the modified relative error bound of MAA is $7u^2\sim 8u^2$. And when the normalization in multiplication is skipped, the overlap in input of addition is $o \approx 3$, the modified relative error bound of Algorithm \ref{algo:sloppyAdd} $\varepsilon_{add} \approx 7u^2$, then the modifed relative error bound of MAA is $11u^2\sim 12u^2$. 
	
	In the evaluations contain $a\cdot b \pm c \cdot d$, for example, hypot function and complex multiplication, the normalization of the multiplications in $a\cdot b$ and $c \cdot d$ can also be skipped.
 
 \subsection{Interval Arithmetic}
 Next, let us focus on the application of this discovery in interval arithmetic. Suppose we are using directed rounding, and all the floating-point operations are rounded with the same direction, then Algorithm \ref{algo:2Prod} is always error-free\cite{HandbookFP2018}, and the ``rounding direction'' of Fast2Sum or 2Sum is the same with rounding direction of floating-point operation\cite{8742617}. These imply that double-word multiplication algorithms and the two addition algorithm in this paper has the same ``rounding direction'' with floating-point operation. Moreover, form Theorem \ref{the:sloppyfast2sum}  and Theorem \ref{the:accuratefast2sum}, even if the input of the two addition algorithms have moderate overlap, this conclusion is still valid. Which means we can safely remove the normalization step in multiplication in MAA , and the ``rounding direction'' of MAA is also the same with floating-point operation. \cite{9370307} point out that Algorithm \ref{algo:accurateAdd} has this property when the inputs do not overlap, but they didn't check the correctness of Fast2Sums in Algorithm  \ref{algo:accurateAdd} in directed rounding.

 	From above analysises, it's unnecesary to switch float-point rounding direction when doing multiplication, addition, or MAA in interval arithmetic. Considering the performance penalty of switching rounding mode, this can enhance the double-word interval arithmetic performance on the CPUs or accelerators which do not support static rounding. In addition, although $s_h$ and $w$ in Fast2Sum of normalization step in addition algorithms may violate the condition $|s_h|\geq w$ because of directed rounding or input's overlap, the Fast2Sum keeps valid. And it's not necessary to substitute 2Sum for Fast2Sum defensively. Anyhow, for CPUs or Accelerators support static rounding, for example, CPUs equipped with AVX512F, it's preferable to do Fast2Sum and 2Sum in round-to-nearest mode to get tighter inclusion, just like the kv library \cite{KV2022}.

\section{Numerical Examples}\label{sec:numeric}	

Matrix multiplication is widely used as a performance benchmark of MMA (FMA).
	By delicate blocking and packing techniques, the performance of double-precision matrix multiplication can be close to the theoretical peak. However, in order to achieve such high performance, the matrix multiplication code has to been carefully optimized and tuned, which highly depends on the experience of developers.  Considering the SIMDization, we use structure of arrays instead of array of structures, just as the practice in \cite{DDAVX,9603410}. Unfortunately this memory layout prevents using of many existing template libraries and tools, such as C++ template library Eigen\cite{eigenweb}. We decided to write double-double GEMM ourselves. To avoid the tedious handwork, we only performed matrix multiplication on small matrices, thus we can keep the matrices in cache by multiple warm-ups instead of blocking and packing. During the test, we traversed all reasonable matrices sizes, repeated it plenty of times on each sizes, and recorded the best FLOPS. We also coded a naive double-precision GEMM without blocking and packing, and obtained 65.4 GFLOPS by above method, close to 67.7 GFLOPS obtained by Intel Math Kernel Library 2019. Therefore, we believe that this method can get a good performance estimation of complete GEMM with blocking and packing.
	
	As 3 fewer floating-point operations are required by Fast2Sum than 2Sum, one can replace 2Sum with Fast2Sum following a comparsion. Normally this is not recommended because of the performance penalty of wrong branch prediction when comparing on modern CPUs. However, the Intel AVX512DQ instruction set introduces vrangepd, an instruction returns MAX/MIN of the operands by absolute value directly, making it possible to eliminating the conditional branch. Fast2Sum combined with vrangepd only requires 5 operations, in which MAX, MIN, $\circ(a + b)$ are independent, thus exposes instruction-level parallelism. Hopefully we can get better performance by replace 2Sum with Fast2Sum conbined with vrangepd. Therefore, we add the test where 2Sum is replaced by vrangepd and Fast2Sum.  
	
	Algorithm 11 in \cite{JM2018DoubleDoubleError} is used for double-double multiplication in MAA.

	.\cite{7384331} implement an accurate double GEMM with the help of compensated inner product. Though compensation techinque cannot be applied to double-word input, its equivalent working precision can be far higher than double, and has decent performance. So we also implemented a compensated GEMM and added it to the test. In the end, we tested the performance of double-precision GEMM for comparison, specifically, performance of function cblas\_degemm in MKL(Row ``MKL dgemm'' in Table \ref{tab:xeon}) and naive GEMM without blocking and packing(Row ``handwritten dgemm'' in Table \ref{tab:xeon}).

	\begin{table}[H]
		\begin{center} 
			\caption{GEMM performance on Intel(R) Xeon(R) Gold 6148. compiler g++ 9.3.0. optimization option: -march=skylake-avx512 -O3.}
            \label{tab:xeon}
			\resizebox{\linewidth}{!}{ 
			\begin{tabular}{cccccccc}
				\toprule
			\makecell[c]{omit add \\normalization} &  	\makecell[c]{omit mul\\normalization} &	\makecell[c]{ sloppy\\ add} & \#comparison & \#add & \#mul & GFLOPS  & 	\makecell[c]{GFLOPS\\(double)}\\
			   \midrule
			no  &  no  & no  & 0 & 26 & 4 & 2.17 & 56.4 \\   	
			no  &  no  & yes & 0 & 17 & 4 & 3.39 & 57.6 \\  	
			no  &  yes & no  & 0 & 23 & 4 & 2.50 & 57.5 \\  
			no  &  yes & yes & 0 & 14 & 4 & 4.20 & 58.8 \\ 
			no  &  no  & no  & 4 & 20 & 4 & 2.48 & 59.5 \\   	
			no  &  no  & yes & 2 & 14 & 4 & 3.61 & 57.8 \\  	
			no  &  yes & no  & 4 & 17 & 4 & 2.87 & 60.3 \\  
			no  &  yes & yes & 2 & 11 & 4 & 4.57 & 59.4 \\ 
			\midrule
			yes &  yes & no  & 0 & 20 & 4 & 2.99 & 59.8 \\  
			yes &  yes & yes & 0 & 11 & 4 & 5.26 & 57.9 \\ 	
		    yes &  yes & no  & 4 & 14 & 4 & 3.42 & 61.6 \\  
			yes &  yes & yes & 2 & 8  & 4 & 5.79 & 57.9 \\ 
			\midrule
		 \multicolumn{3}{c}{compensated dot product} & 0 & 9 & 2 & 6.68 & 60.1 \\
		 \multicolumn{3}{c}{compensated dot product} & 2 & 6 & 2 & 7.66 & 61.3 \\
		    \midrule
		 \multicolumn{3}{c} {handwritten dgemm} & 0 & 1 & 1 & 65.4 & 65.4 \\
		 \multicolumn{3}{c} {MKL dgemm} & 0 & 1 & 1 & 67.7 & 67.7 \\
			 \bottomrule
			\end{tabular}
			}
		\end{center}
	
	\end{table}
	
	The test result is shown as Table \ref{tab:xeon}.  The 1st and 2nd columns indicate whether the normalizations in the addition and multiplication of MAA are skipped, respectively.  The 3rd column indicates the addition algorithm in MAA: ``yes'' means Algorithm  \ref{algo:sloppyAdd} and ``no'' means Algorithm \ref{algo:accurateAdd}. The 4th column is the number of operations of MAX/MIN by absolute value, i.e., vrangepd instructions.  The 5th and 6th columns are the number of double-precision addition and multiplication operations, respectively. And an FMA operation is counted as a multiplication and an addition. The GFLOPS(double) column shows the equivalent double-precision FLOPS for the various algorithms. As there are more double-precision addition and subtraction operations than multiplications in the double-word MAA, pair arithmetic MAA, or compensated GEMM, and the throughput of double-precision multiplication and addition are equal in the CPU for test\cite{instrinsicsguide}, so the equivalent double-precision FLOPS is calculated by multiplying \#add with number in the column GFLOPS. The \#compare column is the number of vrangepd instructions, nonzero value suggests 2Sum is replaced with combination of vrangepd and Fast2Sum. Considering vrangepd and addition instruction share the same execution ports\cite{xeonarch}, latency and throughput\cite{instrinsicsguide}, vrangepd instructions are treated as additions while calculating equivalent GFLOPS(double).
	 
	Considering the reproducibility, we fixed the CPU frequency to 2.4 GHz and AVX-512 frequency to 2.2 GHz. cpufp\cite{cpufp} is a tool that can accurately measure the CPU FLOPS written in assembly language. By this tool, we obatined a double-precision performance of $70.184$ GFLOPS on test CPU, which is very close to the theoretical  peak $2.2\times 8 \times 2 \times 2 = 70.4$GFLOPS, where 2.2 is the AVX-512 frequency, 8 means the width of AVX-512 registers, and the following double 2s are the number of execution ports supporting FMA as well as the number of floating-point operations in each FMA instruction. One can see that the equivalent double-precision GFLOPS for the  multiple-precision algorithms in Table \ref{tab:xeon} are $80.1\% - 87.4\%$ about the theoretical peak 70.4 GLOPS, not as good as $96.2\%$ by MKL. Yet we believe such performance is  acceptable, as the data dependency in these algorithms impair instruction level parallelism, and more registers are required than double-precision, such that it is more difficult to fully exploit the floating-point capacity of CPU. 
	
	One can see that in double-double GEMM, vrangepd can bring $6\%\sim 15\%$ performance gain(GFLOPS). Also, by replacing accurate addition with sloppy addition and then skipping normalization in multiplication, additional $(\frac {4.57} {2.48} - 1) \times 100\%\approx 84\%$ performance gain can be obtained. vrangepd also contributes $10\%\sim 15\%$ performance boost in pair arithmetic GEMM and compensated GEMM.

  \begin{table}
 	\begin{center} \small
 		\caption{Modified relative error of MAA.} 
            \label{tab:error}
 			\resizebox{\linewidth}{!}{ 
 		\begin{tabular}{cccccccc}
 			\toprule
 			\multirow{2}{*}{\makecell[c]{omit mul \\normalization} } & \multirow{2}{*}{\makecell[c]{sloppy\\add}}  & \multicolumn{3}{c}{average} & \multicolumn{3}{c}{max} \\
 			&                                 & average & std  & max  & average & std   & max  \\
 			\midrule
 			no  & no  &  2.20e-33 & 1.84e-36 & 2.21e-33 & 3.76e-32 & 2.73e-33 & 4.35e-32 \\
 			no  & yes &  2.53e-33 & 1.79e-36 & 2.53e-33 & 3.76e-32 & 2.73e-33 & 4.35e-32 \\
 			yes & no  &  2.22e-33 & 1.80e-36 & 2.22e-33 & 4.25e-32 & 5.90e-33 & 5.35e-32 \\
 			yes & yes &  2.60e-33 & 1.91e-36 & 2.60e-33 & 5.15e-32 & 6.92e-33 & 6.03e-32 \\
 			\bottomrule
 		\end{tabular}
 	}
 	\end{center}
 \end{table}
	 
	 The error of four different MAA algorithms is shown in Table \ref{tab:error}. The test input are arrays $a$, $b$, and $c$ with length $n = 10^6$, and all array elements are random numbers uniformly distributed in $[-1/2, 1/2]$. 
     Since the statistical of relative error is ofen interfered by outliers, i.e., several orders of magnitude greater than average relative errors due to cancellation, we use modifed relative error as an alternative. The test is repeated 10 times to confirm the reliability. 	 
	 
	 From Table \ref{tab:error}, we conclude that skipping the normalization in multiplication or using sloppy addition algorithm instead of accurate one has a marginal effect on error of MAA. Developers can safely enhance the performance by these two methods.
	  
	\section{Conclusion}
    
We have demonstrated the effectiveness of both the sloppy and accurate double-word addition algorithms, showing that they perform well when there is moderate overlap in input under faithful rounding mode. The accurate algorithm still guarantees a relative error bound with only minor changes, showcasing their robustness beyond general expectations. The MAA (FMA) is a fundamental component of linear algebra libraries such as BLAS. By omitting the normalization step in multiplication and replacing accurate addition with sloppy addition, the performance of double-double MAA can be nearly doubled without sacrificing precision. Additionally, we found that vrangepd in the AVX512DQ instruction set can provide a performance boost to double-double GEMM or compensated GEMM without any additional cost. Therefore, we recommend replacing 2Sum with vrangepd+Fast2Sum in machines supporting the AVX512DQ instruction set.

In the case of directed rounding, these double-word algorithms, particularly the two addition algorithms, multiplication algorithm (Algorithm \ref{algo:unifymult}), and MAA, exhibit the desired property: their "rounding directions" remain consistent with floating-point operations. As a result, it is no longer necessary to switch rounding direction in interval arithmetic, reducing the burden of implementing interval arithmetic algorithms and improving the performance of interval arithmetic.

\bibliographystyle{unsrtnat}
\bibliography{main}  

\begin{thebibliography}{41}
\providecommand{\natexlab}[1]{#1}
\providecommand{\url}[1]{\texttt{#1}}
\expandafter\ifx\csname urlstyle\endcsname\relax
  \providecommand{\doi}[1]{doi: #1}\else
  \providecommand{\doi}{doi: \begingroup \urlstyle{rm}\Url}\fi

\bibitem[Higham(2002)]{higham2002accuracy}
Nicholas~J Higham.
\newblock \emph{Accuracy and Stability of Numerical Algorithms}.
\newblock Society for Industrial and Applied Mathematics, 2002.

\bibitem[Overton(2001)]{overton2001numerical}
Michael~L Overton.
\newblock Numerical computing with ieee floating point arithmetic.
\newblock In \emph{Software for numerical mathematics}, pages 233--251. Springer, 2001.

\bibitem[Gustafson(1997)]{gustafson1997beating}
John~L Gustafson.
\newblock Beating floating point at its own game: Posit arithmetic.
\newblock \emph{Supercomputing Review}, 9\penalty0 (4):\penalty0 12--15, 1997.

\bibitem[Monniaux(2008)]{monniaux2008certified}
David Monniaux.
\newblock Certified error bounds for expressions involving rounded operators.
\newblock \emph{ACM Transactions on Mathematical Software (TOMS)}, 32\penalty0 (2):\penalty0 236--251, 2008.

\bibitem[Goldberg(1991)]{goldberg1991what}
David Goldberg.
\newblock What every computer scientist should know about floating-point arithmetic.
\newblock \emph{ACM Computing Surveys (CSUR)}, 23\penalty0 (1):\penalty0 5--48, 1991.

\bibitem[Moler(2004)]{moler2004floating}
Cleve Moler.
\newblock Floating points.
\newblock \emph{MATLAB News \& Notes}, 4\penalty0 (2):\penalty0 8--16, 2004.

\bibitem[Hubbert(2012)]{hubbert2012understanding}
J.~C. Hubbert.
\newblock Understanding and using ieee-754 floating point.
\newblock \emph{IBM Journal of Research and Development}, 32\penalty0 (1):\penalty0 74--82, 2012.

\bibitem[Knuth(1998)]{TAOCP}
Donald~E. Knuth.
\newblock \emph{{The Art of Computer Programming}}, volume~2.
\newblock Addison-Wesley, Reading, Massachusetts, 3 edition, 1998.

\bibitem[M\o{}ller(1965)]{10.1007/BF01975722}
Ole M\o{}ller.
\newblock {Quasi double-precision in floating point addition}.
\newblock \emph{BIT}, 5\penalty0 (1):\penalty0 37–50, mar 1965.
\newblock ISSN 0006-3835.
\newblock \doi{10.1007/BF01975722}.
\newblock URL \url{https://doi.org/10.1007/BF01975722}.

\bibitem[Dekker(1971)]{Dekker1971Split}
T.~J. Dekker.
\newblock {A floating-point technique for extending the available precision}.
\newblock \emph{Numerische Mathematik}, 18:\penalty0 224--242, 1971.
\newblock URL \url{https://cir.nii.ac.jp/crid/1571980074212991488}.

\bibitem[IEE(2008)]{IEEE7542008}
{IEEE Standard for Floating-Point Arithmetic}.
\newblock \emph{IEEE Std 754-2008}, pages 1--70, 8 2008.
\newblock \doi{10.1109/IEEESTD.2008.4610935}.

\bibitem[Bailey(2005)]{1425396}
D.H. Bailey.
\newblock High-precision floating-point arithmetic in scientific computation.
\newblock \emph{Computing in Science \& Engineering}, 7\penalty0 (3):\penalty0 54--61, 2005.
\newblock \doi{10.1109/MCSE.2005.52}.

\bibitem[Kouya(2021{\natexlab{a}})]{9603410}
Tomonori Kouya.
\newblock {Acceleration of LU decomposition supporting double-double, triple-double, and quadruple-double precision floating-point arithmetic with AVX2}.
\newblock In \emph{2021 IEEE 28th Symposium on Computer Arithmetic (ARITH)}, pages 54--61, June 2021{\natexlab{a}}.
\newblock \doi{10.1109/ARITH51176.2021.00021}.

\bibitem[Kouya(2021{\natexlab{b}})]{10.1007/978-3-030-86976-2_14}
Tomonori Kouya.
\newblock {Acceleration of Multiple Precision Matrix Multiplication Based on Multi-Component Floating-Point Arithmetic Using AVX2}.
\newblock In \emph{Computational Science and Its Applications – ICCSA 2021: 21st International Conference, Cagliari, Italy, September 13–16, 2021, Proceedings, Part V}, page 202–217, Berlin, Heidelberg, 2021{\natexlab{b}}. Springer-Verlag.
\newblock ISBN 978-3-030-86975-5.
\newblock \doi{10.1007/978-3-030-86976-2_14}.
\newblock URL \url{https://doi.org/10.1007/978-3-030-86976-2_14}.

\bibitem[Hishinuma(2021)]{DDAVX}
Toshiaki Hishinuma.
\newblock {DD-AVX Library: Library of High Precision Sparse Matrix Operations Accelerated by SIMD}, 2021.
\newblock URL \url{https://github.com/t-hishinuma/DD-AVX_v3}.

\bibitem[Rivera et~al.(2021)Rivera, Franchetti, and Püschel]{9370307}
Joao Rivera, Franz Franchetti, and Markus Püschel.
\newblock {An Interval Compiler for Sound Floating-Point Computations}.
\newblock In \emph{2021 IEEE/ACM International Symposium on Code Generation and Optimization (CGO)}, pages 52--64, 2021.
\newblock \doi{10.1109/CGO51591.2021.9370307}.

\bibitem[Wilkinson(1963)]{wilkinson1963rounding}
James~H Wilkinson.
\newblock Rounding errors in algebraic processes.
\newblock \emph{Notes on Applied Science}, 32\penalty0 (1):\penalty0 41--74, 1963.

\bibitem[Olver(2014)]{olver2014numerical}
Sheehan Olver.
\newblock \emph{Numerical Computing with IEEE Floating Point Arithmetic}.
\newblock SIAM, 2014.

\bibitem[Joldes et~al.(2017)Joldes, Muller, and Popescu]{JM2018DoubleDoubleError}
Mioara Joldes, Jean-Michel Muller, and Valentina Popescu.
\newblock {Tight and Rigorous Error Bounds for Basic Building Blocks of Double-Word Arithmetic}.
\newblock \emph{ACM Trans. Math. Softw.}, 44\penalty0 (2), 10 2017.
\newblock ISSN 0098-3500.
\newblock \doi{10.1145/3121432}.
\newblock URL \url{https://doi.org/10.1145/3121432}.

\bibitem[Sterbenz(1974)]{SterbenzLemma}
Pat~H. Sterbenz.
\newblock \emph{{Floating-point computation}}.
\newblock Prentice-Hall series in automatic computation. Prentice-Hall, 1974.
\newblock URL \url{https://cir.nii.ac.jp/crid/1130282270280634752}.

\bibitem[Graillat and Jézéquel(2020)]{8742617}
Stef Graillat and Fabienne Jézéquel.
\newblock {Tight Interval Inclusions with Compensated Algorithms}.
\newblock \emph{IEEE Transactions on Computers}, 69\penalty0 (12):\penalty0 1774--1783, 2020.
\newblock \doi{10.1109/TC.2019.2924005}.

\bibitem[Demmel and Nguyen(2013)]{6545904}
James Demmel and Hong~Diep Nguyen.
\newblock {Fast Reproducible Floating-Point Summation}.
\newblock In \emph{2013 IEEE 21st Symposium on Computer Arithmetic}, pages 163--172, 2013.
\newblock \doi{10.1109/ARITH.2013.9}.

\bibitem[Boldo et~al.(2017)Boldo, Graillat, and Muller]{10.1145/3054947}
Sylvie Boldo, Stef Graillat, and Jean-Michel Muller.
\newblock On the robustness of the 2sum and fast2sum algorithms.
\newblock \emph{ACM Trans. Math. Softw.}, 44\penalty0 (1), jul 2017.
\newblock ISSN 0098-3500.
\newblock \doi{10.1145/3054947}.
\newblock URL \url{https://doi.org/10.1145/3054947}.

\bibitem[Kahan(1997)]{1570854174952231424}
W.~Kahan.
\newblock {Lecture Notes on the Status of IEEE Standard 754 for Binary Floating-Point Arithmetic}.
\newblock 1997.
\newblock URL \url{https://people.eecs.berkeley.edu/~wkahan/ieee754status/IEEE754.PDF}.

\bibitem[Nievergelt(2003)]{10.1145/641876.641878}
Yves Nievergelt.
\newblock {Scalar fused multiply-add instructions produce floating-point matrix arithmetic provably accurate to the penultimate digit}.
\newblock \emph{ACM Trans. Math. Softw.}, 29\penalty0 (1):\penalty0 27–48, mar 2003.
\newblock ISSN 0098-3500.
\newblock \doi{10.1145/641876.641878}.
\newblock URL \url{https://doi.org/10.1145/641876.641878}.

\bibitem[Muller et~al.(2018)Muller, Brunie, de~Dinechin, Jeannerod, Joldes, Lefvre, Melquiond, Revol, and Torres]{HandbookFP2018}
Jean-Michel Muller, Nicolas Brunie, Florent de~Dinechin, Claude-Pierre Jeannerod, Mioara Joldes, Vincent Lefvre, Guillaume Melquiond, Nathalie Revol, and Serge Torres.
\newblock \emph{{Handbook of Floating-Point Arithmetic}}.
\newblock Birkh\"{a}user Basel, 2nd edition, 2018.
\newblock ISBN 3319765256.

\bibitem[Hida et~al.(2001)Hida, Li, and Bailey]{930115}
Y.~Hida, X.S. Li, and D.H. Bailey.
\newblock {Algorithms for quad-double precision floating point arithmetic}.
\newblock In \emph{Proceedings 15th IEEE Symposium on Computer Arithmetic. ARITH-15 2001}, pages 155--162, 2001.
\newblock \doi{10.1109/ARITH.2001.930115}.

\bibitem[Graillat et~al.(2015)Graillat, Lauter, Tang, Yamanaka, and Oishi]{10.1145/2699469}
Stef Graillat, Christoph Lauter, PING Tak~Peter Tang, Naoya Yamanaka, and Shin’ichi Oishi.
\newblock {Efficient Calculations of Faithfully Rounded L2-Norms of n-Vectors}.
\newblock \emph{ACM Trans. Math. Softw.}, 41\penalty0 (4), 10 2015.
\newblock ISSN 0098-3500.
\newblock \doi{10.1145/2699469}.
\newblock URL \url{https://doi.org/10.1145/2699469}.

\bibitem[Rump et~al.(2008)Rump, Ogita, and Oishi]{Rump2008AccurateSumFaith}
Siegfried~M. Rump, Takeshi Ogita, and Shin'ichi Oishi.
\newblock {Accurate Floating-Point Summation Part I: Faithful Rounding}.
\newblock \emph{SIAM Journal on Scientific Computing}, 31\penalty0 (1):\penalty0 189--224, 2008.
\newblock \doi{10.1137/050645671}.
\newblock URL \url{https://doi.org/10.1137/050645671}.

\bibitem[Lauter(2005)]{lauter2005tripledouble}
Christoph~Quirin Lauter.
\newblock {Basic building blocks for a triple-double intermediate format}.
\newblock Research Report RR-5702, LIP RR-2005-38, {INRIA, LIP}, 9 2005.
\newblock URL \url{https://hal.inria.fr/inria-00070314}.

\bibitem[Lef\`{e}vre et~al.(2022)Lef\`{e}vre, Louvet, Muller, Picot, and Rideau]{Vicent2022DDNrm2}
Vincent Lef\`{e}vre, Nicolas Louvet, Jean-Michel Muller, Joris Picot, and Laurence Rideau.
\newblock {Accurate Calculation of Euclidean Norms Using Double-Word Arithmetic}.
\newblock \emph{ACM Trans. Math. Softw.}, oct 2022.
\newblock ISSN 0098-3500.
\newblock \doi{10.1145/3568672}.
\newblock URL \url{https://doi.org/10.1145/3568672}.
\newblock Just Accepted.

\bibitem[Shibata and Petrogalli(2020)]{SLEEF2020}
Naoki Shibata and Francesco Petrogalli.
\newblock {SLEEF: A Portable Vectorized Library of C Standard Mathematical Functions}.
\newblock \emph{IEEE Transactions on Parallel and Distributed Systems}, 31\penalty0 (6):\penalty0 1316--1327, 6 2020.
\newblock ISSN 1558-2183.
\newblock \doi{10.1109/TPDS.2019.2960333}.

\bibitem[Lange and Rump(2020)]{Rump2020PairNIC}
Marko Lange and Siegfried~M. Rump.
\newblock {Faithfully Rounded Floating-Point Computations}.
\newblock \emph{ACM Trans. Math. Softw.}, 46\penalty0 (3), 7 2020.
\newblock ISSN 0098-3500.
\newblock \doi{10.1145/3290955}.
\newblock URL \url{https://doi.org/10.1145/3290955}.

\bibitem[Demmel et~al.(2008)Demmel, Dumitriu, Holtz, and Koev]{Demmel2008NIC}
James Demmel, Ioana Dumitriu, Olga Holtz, and Plamen Koev.
\newblock {Accurate and Efficient Expression Evaluation and Linear Algebra}.
\newblock \emph{Acta Numerica}, 17, 01 2008.
\newblock \doi{10.1017/S0962492906350015}.

\bibitem[Muller and Rideau(2022)]{JM2022DoubleDoubleErrorFormal}
Jean-Michel Muller and Laurence Rideau.
\newblock {Formalization of Double-Word Arithmetic, and Comments on ``Tight and Rigorous Error Bounds for Basic Building Blocks of Double-Word Arithmetic''}.
\newblock \emph{ACM Trans. Math. Softw.}, 48\penalty0 (1), 2 2022.
\newblock ISSN 0098-3500.
\newblock \doi{10.1145/3484514}.
\newblock URL \url{https://doi.org/10.1145/3484514}.

\bibitem[Kashiwagi(2022)]{KV2022}
Masahide Kashiwagi.
\newblock {kv-a C++ Library for Verified Numerical Computation}, 2022.
\newblock URL \url{http://verifiedby.me/kv/index-e.html}.

\bibitem[Guennebaud et~al.(2010)Guennebaud, Jacob, et~al.]{eigenweb}
Ga\"{e}l Guennebaud, Beno\^{i}t Jacob, et~al.
\newblock Eigen v3, 2010.
\newblock URL \url{http://eigen.tuxfamily.org}.

\bibitem[Jiang et~al.(2015)Jiang, Wang, Li, Yang, Zhao, and Huang]{7384331}
Hao Jiang, Feng Wang, Kuan Li, Canqun Yang, Kejia Zhao, and Chun Huang.
\newblock {Implementation of an Accurate and Efficient Compensated DGEMM for 64-bit ARMv8 Multi-Core Processors}.
\newblock In \emph{2015 IEEE 21st International Conference on Parallel and Distributed Systems (ICPADS)}, pages 491--498, 2015.
\newblock \doi{10.1109/ICPADS.2015.68}.

\bibitem[{Intel Corporation}(2023)]{instrinsicsguide}
{Intel Corporation}.
\newblock {Intel Intrinsics Guide}, 2023.
\newblock URL \url{https://www.intel.com/content/www/us/en/docs/intrinsics-guide/index.html}.

\bibitem[Akhilesh~Kumar(2017)]{xeonarch}
Malay~Trivedi Akhilesh~Kumar.
\newblock {Intel Xeon Scalable Processor Architecture Deep Dive}.
\newblock In \emph{Intel Press Workshops}, 2017.

\bibitem[Gao(2024)]{cpufp}
Yang Gao.
\newblock {A CPU tool for benchmarking the floating-points peak performance}, 2024.
\newblock URL \url{https://github.com/pigirons/cpufp}.

\end{thebibliography}






\end{document}